\documentclass[a4paper, 11pt]{amsart}

\usepackage[T1]{fontenc}
\usepackage[utf8]{inputenc}

\usepackage{amssymb}

\makeatletter
\@namedef{subjclassname@2010}{%
  \textup{2010} Mathematics Subject Classification}
\makeatother

\def\F{\mathrm{F}}
\def\E{\mathrm{End}}
\def\c{\mathop{\mathrm{card}}}
\def\rad{\mathop{\mathrm{rad}}}
\def\int{\mathop{\mathrm{Int}}}
\allowdisplaybreaks[2]

\newtheorem{theorem}{Theorem}[section]
\newtheorem{proposition}[theorem]{Proposition}
\newtheorem{lemma}[theorem]{Lemma}
\newtheorem{corollary}[theorem]{Corollary}
\theoremstyle{definition}
\newtheorem{remark}[theorem]{Remark}
\newtheorem{definition}[theorem]{Definition}
\newtheorem{example}[theorem]{Example}

\numberwithin{equation}{section}

\begin{document}

\title{Centrally Stable Algebras}

\begin{abstract}
We define an algebra $A$ to be centrally stable if, for every epimorphism $\varphi$ from $A$ to another algebra $B$, the center $Z(B)$ of $B$ is equal to $\varphi(Z(A))$, the image of the center of $A$. After providing some examples and basic observations, we consider in  somewhat greater detail central stability in tensor products of algebras, and finally establish our main result which states  that a finite-dimensional unital algebra $A$ over a perfect field $F$ is centrally stable if and only if $A$ is isomorphic to a direct product of algebras of the form $C_i\otimes_{F_i}A_i$, where  $F_i$ is a field extension of $F$, $C_i$ is a commutative $F_i$-algebra, and $A_i$ is a central simple $F_i$-algebra.
\end{abstract}

\author{Matej Bre\v sar}
\author{Ilja Gogi\'{c}}

\address{M. Bre\v sar,  Faculty of Mathematics and Physics,  University of Ljubljana,
 and Faculty of Natural Sciences and Mathematics, University 
of Maribor, Slovenia} \email{matej.bresar@fmf.uni-lj.si}

\address{I. Gogi\'{c}, Department of Mathematics, University of Zagreb,
Zagreb, Croatia}

\email{ilja@math.hr}

\thanks{The first named author was supported by the Slovenian Research Agency (ARRS) Grant P1-0288. 
The second named author was supported by 
 the Croatian Science Foundation under the project IP-2016-06-1046.}

\keywords{Centrally stable algebra, centrally stable element, center, radical, tensor product of algebras, finite-dimensional algebra}

\subjclass[2010]{16D70, 16K20, 16N20, 16U70}
\maketitle

\section{Introduction}

Let $A$ and $B$ be algebras over a field with centers $Z(A)$ and $Z(B)$. If $\varphi:A\to B$ is an epimorphism, then  $\varphi(Z(A))$ is obviously contained in, but is not necessarily equal to $Z(B)$.
In 1971, Vesterstr\o m \cite{V} proved that  if $A$ is a unital $C^*$-algebra, then every $*$-epimorphism $\varphi$  from $A$ to another unital  $C^*$-algebra $B$ satisfies $\varphi(Z(A)) = Z(B)$ if and only if $A$ is weakly central (i.e., for any pair of maximal ideals  $M$ and $N$  of $A$, 
$M \cap Z(A) =N \cap Z(A)$ implies $M=N$ 
\cite{Mis}). This result is both beautiful and useful, and therefore gives rise to a question whether   algebras satisfying such a condition can be described in some other contexts. The setting in which we will work is purely algebraic and not restricted to algebras with involution. 
We propose to consider the following class of algebras.

\begin{definition} \label{d1}An {\em algebra}
 $A$ is {\em centrally stable} if, for every algebra epimorphism $\varphi:A\to B$, 
$\varphi(Z(A))=Z(B)$. 
\end{definition}
Thus, speaking loosely, homomorphic images of a centrally stable algebra $A$ cannot have larger centers than $A$. It should be mentioned that  $C ^*$-algebras satisfying an analogous condition appear under a different name, namely, $C ^*$-algebras having
  the {\em center-quotient property} (see, e.g., \cite[p. 2671]{ALT}).

We also introduce a local version of Definition \ref{d1}.

\begin{definition} \label{d2} An {\em element} $a$ of an algebra  $A$
is  {\em centrally stable} if,  for every algebra epimorphism $\varphi:A\to B$, $\varphi(a)\in Z(B)$ implies $a\in Z(A) +\ker\varphi$.
\end{definition}

Note that $A$ is a centrally stable algebra if and only if every $a\in A$  
is a centrally stable element.
We remark that both definitions obviously also  make sense for rings, but it will more convenient for us to work with algebras over  fields.

The paper is organized as follows. In Section \ref{s2}, we record various basic facts and examine several examples. In Section \ref{s3}, we study  central stability of elements in the tensor product of an arbitrary unital algebra with a central simple algebra. 
The main result is proved in Section \ref{s4}. We show that a finite-dimensional unital algebra $A$ over a perfect field $F$ is centrally stable if and only if 
$$A\cong (C_1\otimes_{F_1} A_1)\times \cdots\times  (C_r\otimes_{F_r} A_r),$$
 where each $F_i$ is a field extension of $F$,
$C_i$ is a commutative  $F_i$-algebra, and $A_i$ is a central simple $F_i$-algebra.
The proof uses the classical theory of finite-dimensional algebras, including  Wedderburn's structure theory, the Skolem-Noether Theorem, and the   Artin-Whaples Theorem.

\section{Remarks and Examples} \label{s2}

We first introduce some notation. Throughout, $F$ denotes a field and, unless specified otherwise, our algebras will be algebras over $F$.    
Given any algebra $A$, we write $Z(A)$ for the center of $A$. If $X$ is a subset of $A$, then 
 Id$(X)$ denotes the ideal of $A$ generated by $X$. As usual, $[x,y]$ stands for $xy-yx$.
 
The next proposition presents some alternative definitions of  centrally stable algebras.

\begin{proposition}\label{ldef}
Let $A$ be an algebra. The following conditions are equivalent:
\begin{enumerate}
\item[{\rm (i)}] $A$ is centrally stable.
\item[{\rm (ii)}] For every ideal $I$ of $A$, $Z(A/I)=(Z(A)+I)/I$.
\item[{\rm (iii)}] For every $a\in A$, $a\in Z(A) +{\rm Id}([a,A])$.
\end{enumerate}
\end{proposition}

\begin{proof} (i)$\implies$(ii). Consider the canonical epimorphism $\varphi:A\to A/I$.

(ii)$\implies$(iii). Take ${\rm Id}([a,A])$ for $I$ in (ii).

(iii)$\implies$(i). Let  $\varphi:A\to B$ be an epimorphism and let $\varphi(a)\in Z(B)$. Then $[a,A]\subseteq \ker\varphi$ and so, by (iii), $a\in Z(A) +\ker\varphi$. Thus, $\varphi(a)\in \varphi(Z(A))$.
\end{proof}

From (iii) we see that a necessary condition for $A$ to be centrally stable is that it is equal to the sum of its center and its commutator ideal.

A similar proposition holds for centrally stable elements. The proof requires only obvious changes, so we omit it.

\begin{proposition}\label{pe}
Let $A$ be an arbitrary algebra and let $a\in A$. The following conditions are equivalent:
\begin{enumerate}
\item[{\rm (i)}] $a$ is  centrally stable.
\item[{\rm (ii)}] For every ideal $I$ of $A$, $a+I\in Z(A/I)$ implies $a\in Z(A)+I$.
\item[{\rm (iii)}]  $a\in Z(A) + {\rm Id}([a,A])$.
\end{enumerate}
\end{proposition}

We proceed to examples. First we record the two most obvious ones.

\begin{example}\label{et1}
Every simple algebra is centrally stable.
\end{example}

\begin{example}\label{et2}
Every commutative algebra is centrally stable.
\end{example}

More generally, elements from the center of any algebra $A$  are always centrally stable. It may be that $A$ has no other centrally stable elements.

\begin{example} Let $X$ be a set with at least two elements and let $A=F\langle X\rangle$ be the free algebra on $X$. Then only scalar multiples of $1$ are centrally stable in $A$. Indeed, take $f_0\in A$ that is not a scalar multiple of $1$, and let $m$ be a monomial of $f_0$ of lowest positive degree. Note that for any $f\in A$, 
the commutator $[f_0,f]$ is either $0$ or is a sum of monomials that have higher degree than $m$. But then the same is true for all polynomials in Id$([f_0,A])$. Hence, $f_0\notin Z(A)+{\rm Id}([f_0,A])$, so  $f_0$ is not centrally stable by Proposition \ref{pe}.
\end{example}

By $M_n(A)$ we denote the algebra of all $n\times n$ matrices with entries in the algebra $A$. We write $e_{ij}$ for the standard matrix units in $M_n(A)$.

\begin{example}\label{e33}
Let $S$ be the subalgebra of  $M_n(F)$ consisting of all  matrices that can be written as a sum of a scalar matrix and a strictly upper triangular matrix.
Its center consists of 
 matrices of the form $\lambda 1 + \mu e_{1n}$ where $\lambda,\mu\in F$. It is easy to see that these are the only matrices that satisfy condition (iii) of Proposition \ref{pe}. 
Thus, again, only central elements 
 are centrally stable in $S$ (see also Proposition \ref{prad}).
\end{example}

\begin{example}\label{e33t}
Let  $T$ be the subalgebra of  $M_n(F)$ consisting of all 
upper triangular matrices. Using the fact that every strictly upper triangular matrix $s$ satisfies  $s=\sum_{i=1}^n e_{ii}[e_{ii},s]$ one easily proves that
 $a\in T$ satisfies condition (iii) of Proposition \ref{pe} (that is, $a$ is centrally stable) if and only if 
$a$ 
is   a sum of a scalar matrix and a strictly upper triangular matrix. Thus, the set of centrally stable elements in $T$ is exactly the algebra $S$ from the preceding example, which, however, is not centrally stable for $n\ge 3$.
\end{example}

Our next goal is to provide examples of centrally stable algebras that are not as obvious as those from Examples \ref{et1} and \ref{et2}.
Our main example will be  $\E_F(V)$, the algebra of all linear operators of a (possibly infinite-dimensional) vector space $V$.
We first remark that if $\dim_F V\geq \aleph_0$, then 
every proper ideal of $\E_F(V)$ is of the form 
$$\F_\kappa (V)=\{T \in \E_F(V) : \ \dim_F T(V)<\kappa\}$$
for some cardinal number $\aleph_0 \leq \kappa\leq \dim_F V$ 
(see, e.g., \cite[Corollary 3.4]{GP}).
The following lemma will lead to the desired conclusion.

\begin{lemma}\label{infset}
Let $V$ be an infinite-dimensional vector space over a field $F$ and let $\aleph_0 \leq \kappa \leq \dim V$. If $T  \in \mathrm{End}_F(V)$ is such that $T \notin \F_\kappa(V)+F1$, then 
there exists a subset $\mathcal{S}$ of $V$ of cardinality $\kappa$ such that the set $\mathcal{S}\cup T(\mathcal{S})$ is linearly independent. 
\end{lemma}

\begin{proof}
First note that there is a vector $x \in V$ such that the set $\{x, Tx\}$ is linearly independent. If this was not true, then for each $x \in V$ there would exist $\alpha(x)\in F$ such that  $Tx=\alpha(x)x$. Using the linearity of $T$ it is easy to check that $x \mapsto \alpha(x)$ must be a constant function, so $T$ would be a scalar operator. 

Let $\mathfrak{F}$ be the family of all subsets $\mathcal{S}$ of $V$  with $\c(\mathcal{S})\leq \kappa$ such that the set $\mathcal{S}\cup T(\mathcal{S})$ is linearly independent. As noted in the preceding paragraph, $\mathfrak{F} \neq \emptyset$. We order $\mathfrak{F}$ by the usual set-theoretic inclusion. If $\mathfrak{L}$ is any chain in $\mathfrak{F}$, then obviously $\bigcup_{\mathcal{S}\in \mathfrak{L}} \mathcal{S}$ is an upper bound in $\mathfrak{F}$. Hence, by Zorn's Lemma, there is a maximal element $\mathcal{M}$ of $\mathfrak{F}$. We claim that $\c(\mathcal{M})=\kappa$. Suppose that $\c(\mathcal{M})<\kappa$. Define $W$ to be the $F$-linear span of $\mathcal{M} \cup T(\mathcal{M})$ and choose a subspace $Z$ of $V$ such that 
$V=W \oplus Z$. By maximality of $\mathcal{M}$,  the set $\mathcal{M}\cup T(\mathcal{M}) \cup \{z, Tz\}$ is linearly dependent for any $z \in Z$, so $$Tz \in  W +Fz.$$
 Let $P$ be the idempotent in $\mathrm{End}_F(V)$ with image $Z$ and kernel $W$.
Then for each $x \in V$ there is a scalar $\delta(x)$ such that
$$TPx=(1-P)TPx + \delta(x)Px.$$
Again, using linearity, we conclude that the scalar-valued function $x \mapsto \delta(x)$ must be constant on $Z$; denote this scalar by $\delta$. Then $$TP=(1-P)TP+\delta P,$$ so 
\begin{eqnarray*}
T&=&T(1-P)+TP=T(1-P)+(1-P)TP + \delta P \\
&=&T(1-P)+(1-P)TP -\delta(1-P) +\delta 1 = R+\delta 1
\end{eqnarray*}
where $R=T(1-P)+(1-P)TP -\delta(1-P)$. Since $$\dim_F(1-P)(V)=\dim_F W=\c(\mathcal{M})< \kappa,$$ we arrive at a contradiction  that $R \in \F_\kappa(V)$.
\end{proof}

Recall that an algebra $A$ is said to be {\em central} if it is unital and its center consists of scalar multiples of 
unity. In view of Proposition \ref{ldef}, a central algebra is centrally stable if and only if $A/I$ is a central algebra for every ideal $I$ of $A$.
Note that the algebra $\mathrm{End}_F(V)$  is central.

\begin{proposition}\label{End}
Let $V$ be any vector space over a field $F$. Then
the algebra $\mathrm{End}_F(V)$ is centrally stable.
\end{proposition}

\begin{proof}
If $V$ is finite-dimensional,  then $\E_F(V)$ is simple and so there is nothing to prove. Therefore, suppose that $\dim_F V\geq \aleph_0$. Let $I$ be any ideal of $\E_F(V)$. As already mentioned, there is a cardinal number $\aleph_0 \leq \kappa \leq \dim_F V$ such that $I=\F_\kappa(V)$. We have to prove that the  algebra 
$\mathrm{End}_F(V)/\F_\kappa(V)$ is central. This is equivalent to the following:
$$(\forall T \in \mathrm{End}_F(V)) \, ([T,U]\in \mathrm{F}_\kappa(V) \, \forall U \in \mathrm{End}_F(V) \, \Longrightarrow \, T\in \mathrm{F}_\kappa(V)+F1).$$ In order to show this, suppose that $T \notin \mathrm{F}_\kappa(V)+F1$. By Lemma \ref{infset}, there is a subset $\mathcal{S}$ of $V$ of cardinality $\kappa$ such that the set
$\mathcal{S}\cup T(\mathcal{S})$ is linearly independent. Let $U \in \mathrm{End}_F(V)$ be any linear map such that $Ux=0$ and $U(Tx)=x$ for all $x \in \mathcal{S}$ . Then $[T,U]x=-x$ for all $x \in \mathcal{S}$, which shows that $[T,U]\notin \mathrm{F}_\kappa(V)$.
\end{proof}

\begin{remark}
In contrast to Proposition \ref{End}, if $V$ is a (real or complex) infinite-dimensional Banach space, then the algebra $\mathrm{B}(V)$ of all bounded linear operators on $V$ does not need to be centrally stable. This is due to the fact the center of the quotient algebra $\mathrm{B}(V)/\mathrm{K}(V)$, where $\mathrm{K}(V)$ is the ideal of compact operators, can be quite large, even though $\mathrm{B}(V)$ is central. 
In particular, in \cite{MPZ} it was shown that for each countably infinite compact metric space $X$, there is a Banach space $V$ such
that the Banach algebra $\mathrm{B}(V)/\mathrm{K}(V)$ is isomorphic to the algebra $C(X)$ of scalar-valued continuous functions on $X$. 
\end{remark}

Making use of $\F(V) = \F_{\aleph_0}(V)$, the ideal of  $\E_F(V)$ consisting of all finite rank operators,
 we can produce further examples of centrally stable algebras.

\begin{example} Let $V$ be an infinite-dimensional vector space over $F$ and let 
  $B$ be any  central simple subalgebra
of $\E_F(V)$ that contains the identity operator
(one can for example take one of  Weyl algebras). 
We claim that $A= B+ \F(V)$ is a centrally stable algebra. Indeed,  $B\cap \F(V)=0$ since $B$ is simple and unital, and so $A/\F(V)\cong B$. Therefore, $\F(V)$ is a maximal, and hence a unique ideal of $A$. From the assumption that $B$ is central we conclude that $A$ is centrally stable.
\end{example}

\begin{remark} 
As a partial converse of the preceding example, we can show that if a central algebra $A$ is centrally stable, then $A$ has a unique maximal ideal.
Indeed, assume that 
 $A$ contains two different maximal ideals $M$ and $N$. Then $M+N=A$, so by the Chinese Remainder Theorem, 
$$\varphi : A \to (A/M) \times (A/N) \quad\mbox{given by}\quad \varphi(a)=(a+M,a+N)$$
is an algebra epimorphism. Since $A$ is central and centrally stable, all the centers $Z(A)$, $Z(A/M)$ and $Z(A/N)$ are isomorphic to $F$, while the center of $(A/M) \times (A/N)$ is isomorphic to $F \times F$. Therefore $\varphi$ cannot map $Z(A)$ onto $Z((A/M)\oplus (A/N))$.
\end{remark}

If $A$ is a unital complex Banach algebra and 
$M$ is a maximal ideal of $A$, then $M$ is necessarily closed, so $A/M$ is a simple Banach algebra. Therefore its center consists of scalar multiples of $1$ by the Gelfand-Mazur Theorem. In particular, 
$Z(A/M)=(Z(A)+M)/M$. This is not the case for general algebras.

\begin{example} \label{ema}
Let $K$ be a proper field extension of $F$, and let $A$ be the $F$-algebra consisting of matrices of the form  $\left[\begin{smallmatrix} \alpha & \beta  \cr
0 & \lambda \cr
 \end{smallmatrix} \right]
$ with $\lambda \in F$ and $\alpha,\beta\in K$. The set $M$ of all matrices of the form  $\left[\begin{smallmatrix} 0& \beta  \cr
0 & \lambda \cr
 \end{smallmatrix} \right]
$ is  a maximal ideal of $A$ and $A/M\cong K$. Since the center of $A$ consists of scalar multiples of $1$, $Z(A/M)\ne(Z(A)+M)/M$.
\end{example}

We remark that maximal ideals  played a crucial role in Vesterstr\o m's seminal paper \cite{V}.  Example \ref{ema} indicates that they may not be that useful in the purely algebraic context.

In the rest of this section, we will show that the property of being centrally stable is  preserved  under some algebraic constructions.  We start with homomorphic images.

\begin{proposition}\label{homim} A homomorphic image of a  centrally stable algebra $A$ is centrally stable.
\end{proposition}
\begin{proof}
Let $B$ be a homomorphic image of $A$. Thus, there is an algebra epimorphism $\varphi:A\to B$. Now take an algebra epimorphism $\psi:B\to C$. Since $A$ is centrally stable, it follows that
$$Z(C)= Z\big((\psi\circ\varphi)(A)\big) = (\psi\circ\varphi)(Z(A)) = \psi \big(\varphi(Z(A))\big)= \psi \big(Z(B)\big).$$
This proves that $B$ is centrally stable.  
\end{proof}

Let  $A$ be a  non-unital algebra. Denote by $A^\sharp$ its unitization.
Recall that every element in $A^\sharp$ can be written as $\lambda 1+a$ with $\lambda \in F$ and $a\in A$. 
Note that
$Z(A^\sharp)=F1+Z(A)$ and  
\begin{equation}\label{comidunit}
{\rm Id}([\lambda 1+ a,A^\sharp])={\rm Id}([a,A])\subseteq A.
\end{equation}

\begin{proposition}\label{unit}
A non-unital algebra $A$ is centrally stable if and only if $A^\sharp$
is centrally stable.
\end{proposition} 
\begin{proof}
Assume that $A^\sharp$ is centrally stable and let $a \in A$. 
By Proposition \ref{ldef} (iii),  $a=z+b$ for some $z \in Z(A^\sharp)$ and $b \in {\rm Id}([a,A^\sharp])$. Since $b \in A$ by (\ref{comidunit}), it follows that $z=a-b \in A \cap Z(A^\sharp)=Z(A)$.

Conversely, assume that $A$ is centrally stable and take $\lambda 1 + a \in A^\sharp$. Again using Proposition \ref{ldef} (iii), we have $a\in Z(A)+{\rm Id}([a,A])$, so from   \eqref{comidunit} we conclude that $\lambda 1+ a \in  Z(A^\sharp)+{\rm Id}([\lambda 1 + a,A^\sharp])$.
\end{proof}

The direct sum construction also behaves well with respect to central stability.

\begin{proposition}\label{dirsum} The direct sum of a family of algebras  $\{A_j\}_{j \in J}$ is centrally stable if and only if all algebras $A_j$ are centrally stable.
\end{proposition}
 \begin{proof}
If the direct sum of a family of algebras $\{A_j\}_{j \in J}$ is centrally stable, then  each $A_j$ is centrally stable by Proposition \ref{homim}.

Now suppose  that $\{A_j\}_{j \in J}$ is a family of  centrally stable algebras. Set $A:=\bigoplus_{j \in J} A_j$ and fix an element $a=\{a_j\}\in A$. Then there is a finite subset $S \subseteq J$ such that $a_j=0$ for all $j \in J \setminus S$.
Since each $A_{j}$ is centrally stable, Proposition \ref{ldef} (iii) implies that
each $a_j$, $j\in S$, lies in $Z(A_j)+{\rm Id}([a_j,A_j])$. One easily shows that this implies that $a\in {\rm Id}([a,A]) + Z(A)$.
\end{proof}

The problem of central stability of tensor products is more subtle. It is the topic of the next section.

\section{Centrally Stable Elements in Tensor Products of  Algebras} \label{s3}

We start our investigation of central stability in tensor products with a simple  lemma.

\begin{lemma}\label{lik}
Let $A$ be an arbitrary and  $B$ be  a unital algebra, and let $a\in A$. Then $a$ is centrally stable in $A$ if and only if $a\otimes 1$ is centrally stable in $A\otimes B$.   
\end{lemma}

\begin{proof}
Suppose $a\otimes 1$ is centrally stable in $A\otimes B$.
It is easy to see that $Z(A\otimes B)\subseteq Z(A)\otimes B$ (compare the proof of \cite[Proposition 4.31]{INCA}).
Hence, Proposition \ref{pe} gives
$$a\otimes 1 \in Z(A)\otimes B + {\rm Id}([a\otimes 1, A \otimes B]).$$
Noticing that ${\rm Id}([a\otimes 1, A \otimes B]) = {\rm Id}([a,A])\otimes B$, it follows that $a$ lies in the linear span
of ${\rm Id}([a,A])$ and $Z(A)$. That is, $a$ is centrally stable. 

The converse statement is rather obvious, so we omit the proof. 
\end{proof}

\begin{proposition}\label{l}
Let $A$ and $B$ be unital algebras. 
\begin{itemize}
\item[(a)] If the algebra $A\otimes B$ is centrally stable, then so are  $A$ and $B$. 
\item[(b)] If one of the algebras $A$ or $B$ is centrally stable and the other one is central and simple, then  $A \otimes B$ is centrally stable.
\end{itemize}
\end{proposition}

\begin{proof}
Of course, (a) follows from Lemma \ref{lik}. To prove (b), assume that 
 $A$ is centrally stable and that $B$ is a central simple algebra. Let $L$ be an ideal of
$A \otimes B$. By \cite[Theorem 4.42]{INCA} (and the comment following its proof),  there is an ideal $I$ of $A$ such that $L=I \otimes B$. Let $\pi : A \to A/I$ be the canonical epimorphism. Define $\varphi=\pi \otimes \mathrm{id}_B  : A \otimes B \to (A/I) \otimes B$. Then  
$L =\ker\varphi$, so we can identify $(A\otimes B)/L$ with  $(A/I) \otimes B$.
As $A$ is centrally stable, we have $\pi(Z(A))=Z(A/I)$, so 
\begin{align*} 
\varphi\big(Z(A \otimes B)\big)&=\varphi\big(Z(A)\otimes Z(B)\big)=\pi(Z(A))\otimes Z(B)\\&=Z(A/I)\otimes Z(B)
= Z\big((A/I)\otimes B\big).
\end{align*}
This shows that the algebra $A \otimes B$ is centrally stable.
\end{proof}

Note that (b) is a partial converse of (a). We believe that in general the converse is not true. However, our attempts to find a pair of  centrally stable unital algebras $A$ and $B$ such that $A\otimes B$ is not centrally stable
 were unsuccessful so far (Archbold's result \cite[Theorem 3.1]{A} on a related problem for $C ^*$-algebras indicates the delicacy of this problem).

In the simplest case where $A=M_n(F)$, Proposition \ref{l} states the following.

\begin{corollary}\label{cc}
Let $A$ be a unital algebra and let $n$ be a positive integer. Then $A$ is centrally stable if and only if the matrix algebra $M_n(A)$ is centrally stable. In particular, if $A$ is commutative, then $M_n(A)$ is  centrally stable.
\end{corollary}

Remark \ref{r11} below shows that the last statement does not hold without assuming that $A$ is unital.




In our main result in this section we will consider central stability of elements in  the tensor product of an arbitrary unital algebra $A$ and a central simple algebra $B$ (of arbitrary dimension, finite or infinite). Lemma \ref{lik}
tells us that  $a\otimes 1$ is not centrally stable in $A\otimes B$ if $a$ is not centrally stable in $A$. Of course, elements of the form $a\otimes 1$ may not be 
 the only elements in $A\otimes B$ that are not centrally stable. After all, if $a\otimes 1$ is not centrally stable, then neither is $u(a\otimes 1)u^{-1}$ for every invertible $u\in A\otimes B$. We will show, however, that every element in $A\otimes B$ that is not centrally stable is intimately connected with an element of the form $a\otimes 1$.
More specifically, we will prove the following theorem.

\begin{theorem}\label{t} Let $A$ be an arbitrary unital algebra and $B$ be a central simple algebra.
 For every $t\in T= A\otimes B$ there exist an $a\in A$
and an $s\in T$ such that $t= a\otimes 1 + s$ and  $s$ is centrally stable in $T$.
Moreover, if $a$ is centrally stable in $A$, then $t$ is centrally stable in $T$.
\end{theorem}

In the proof of Theorem \ref{t}, we will make use of the {\em multiplication algebra} of a central simple algebra $B$; recall that this is the subalgebra of the algebra of all linear operators from $B$ to $B$ generated by all left and right multiplication maps $L_b,R_b:B\to B$, $L_b(x)=bx$, $R_b(x)=xb$. We denote it by $M(B)$. The  Artin-Whaples Theorem
states that $M(B)$ is a dense algebra of linear operators of $B$ (see, e.g.,  \cite[Corollary 5.24]{INCA}). This means that, for any linearly independent
$b_1, \dots, b_m\in B$ and arbitrary $c_1,\dots,c_m\in B$, there exists an 
$f\in M(B)$ such that $f(b_i)=c_i$ for all $i=1,\dots,m$.

Another ingredient in the proof is a version of Amitsur's Lemma (see \cite[Theorem 4.2.7]{BMM}) which states
that if $T_1,\dots, T_n$ are linear operators between vector spaces $U$ and $V$ such that the vectors 
$T_1x,\dots,T_nx$ are linearly dependent for every $x\in U$, then a nontrivial linear combination of $T_1, \dots,T_n$ has finite rank (in fact, its rank is at most $n-1$, provided that $F$ is infinite \cite[Theorem 2.2]{BS2}).

Finally, we will need the following lemma, which is a special case of the results from  \cite{B} and \cite{BE}. The proof is short and simple, so we give it here.

\begin{lemma}\label{l2}
Let $d$ be a nonzero derivation of a simple algebra $B$. If $d$ has  finite rank, then $B$ is finite-dimensional.
\end{lemma}

\begin{proof}
Take $a\in B$ such that $d(a)\ne 0$. Since $B$ is simple, $d(a)bd(a)\ne 0$ for some  $b\in B$.
For any $x,y\in B$,
$$xd(a)bd(a)y= \big(d(xa)-d(x)a\big)b \big(d(ay)-ad(y)\big).$$ 
Thus,
$$xd(a)bd(a)y\in (d(B)+d(B)a)b(d(B) + ad(B)).$$
Since $d(B)$ is a finite-dimensional space, it follows that the ideal generated by $d(a)bd(a)$ is finite-dimensional. As $B$ is simple, this gives the desired conclusion.
\end{proof}


\begin{proof}[Proof of Theorem \ref{t}]
We will establish slightly more than stated in the theorem. That is, we will show that  there exist an $a\in A$
and an $s\in T$ such that
\begin{enumerate}
\item[{\rm (a)}] $t= a\otimes 1 + s$.
\item[{\rm (b)}]  $s\in {\rm Id}([s,T])\bigcap  {\rm Id}([t,T])$.
\end{enumerate}
By Proposition \ref{pe}, $s\in {\rm Id}([s,T])$ implies that $s$ is centrally stable in $T$. We will need $s\in {\rm Id}([t,T])$ to prove the last statement of the theorem.

We claim that if $B$ is finite-dimensional, then, in order to find an $a\in A$ and an $s\in T$ satisfying (a) and (b),
 there is no loss of generality in assuming that $B=M_n(F)$. Indeed, let
$B^{{\rm o}}$ be the opposite algebra of $B$. Then $B\otimes B^{{\rm o}}\cong M_n(F)$ where $n$ is the dimension of $B$ \cite[Corollary 4.28]{INCA}, and hence
 $T\otimes B^{{\rm o}}\cong A\otimes M_n(F)$. Assume now that  an $a\in A$ and an $s\in T$ satisfying (a) and (b) exist if $B=M_n(F)$. Take $t\in T$. Then $t\otimes 1 \in T\otimes B^{{\rm o}}$ and so there exist an $a\in A$ and an 
$s'\in T\otimes B^{{\rm o}}$ such that $$t\otimes 1 = a\otimes 1\otimes 1 +s'$$
and 
$$s'\in {\rm Id}([s',T\otimes  B^{{\rm o}}])\bigcap{\rm Id}([t\otimes 1,T\otimes  B^{{\rm o}}]).$$
Set $s= t - a\otimes 1$, so that $s' = s\otimes 1$. Since
 $${\rm Id}([s\otimes 1,T\otimes  B^{{\rm o}}]) = {\rm Id}([s,T])\otimes  B^{{\rm o}}\,\,\,\,\mbox{and}\,\,\,\,
{\rm Id}([t\otimes 1,T\otimes  B^{{\rm o}}]) = {\rm Id}([t,T])\otimes  B^{{\rm o}},$$
it follows that 
$$s\otimes 1 \in {\rm Id}([s,T])\otimes  B^{{\rm o}}\bigcap {\rm Id}([t,T])\otimes  B^{{\rm o}}.$$
But then $s\in {\rm Id}([s,T])\bigcap {\rm Id}([t,T])$, which proves our claim.

Let us now prove that elements 
$a$ and $s$ satisfying (a) and (b) exist for an arbitrary
$$t=\sum_{i=1} ^m a_i\otimes b_i\in T.$$ 
 Without loss of generality, we may assume that $b_1,\dots,b_m$ are linearly independent and $b_1=1$.
Moreover, if  $B$ is finite-dimensional, then we may assume that $B=M_n(F)$ and hence that $m=n^2$ and $b_2,\dots,b_m$ are the standard matrix units $e_{pq}$ with $1\le p,q\le n$ and $(p,q)\ne (n,n)$. 

We claim that for each $i >1$ there exists a $c_i\in B$ such that $[b_i,c_i]$ does not lie in the linear span of $[b_j,c_i]$ with $j\ne i$. If  $B=M_n(F)$ and $b_i = e_{pq}$, then we    take $c_i=e_{qn}$. We may therefore assume that $B$ is 
 infinite-dimensional. Suppose our claim is not true, that is, for every $c\in B$, $[b_i,c]$  lies in the linear span of $[b_j,c]$ with $j\ne i$.  We may then invoke Amitsur's Lemma which tells us that there exists a nontrivial linear combination of the inner derivations $c\mapsto [b_j,c]$, $j=2,\dots,m$, which has finite rank. However, Lemma \ref{l2} implies that it must be equal to zero. This means that 
a nontrivial linear combination of $b_2,\dots,b_m$ lies in the center of $B$, which, by our assumption, 
consists of scalar multiples of $b_1=1$. As this contradicts the assumption that $b_1,b_2,\dots,b_m$ are linearly independent, our claim is proved.

Set $$a=a_1,\,\,\,\, s=t-a\otimes 1,\,\,\,\, I = {\rm Id}([s,T]), \,\,\,\,\mbox{and}\,\,\,\,J = {\rm Id}([t,T]).$$
Our goal is to prove that $s\in I\cap J$.  Fix $i\ge 2$ and let $c_i$ be the element from the preceding paragraph.
We have
$[t,1\otimes c_i] = [s,1\otimes c_i] \in I\cap J,$
that is,
$$ a_2\otimes [b_2,c_i] + \cdots +  a_i\otimes [b_i,c_i] + \cdots + a_m \otimes [b_m,c_i]\in I\cap J.$$  
By the Artin-Whaples Theorem, there exists an $f= \sum_k L_{u_k}R_{v_k}\in M(B)$ such that 
$f([b_i,c_i]) =b_i$ and $f([b_j,c_i]) =0$ for $j\ne i$. Hence,
$$ I\cap J \ni \sum_k (1\otimes u_k) \Big(\sum_{j=2}^m a_j\otimes [b_j,c_i]\Big) (1\otimes v_k)=a_i\otimes b_i.$$
Consequently, $$s=\sum_{i=2}^m a_i\otimes b_i\in I\cap J,$$ as desired.


Finally, assume that $a$ is centrally stable in $A$. By Proposition \ref{pe}, there exist an  $\alpha\in Z(A)$ and $x_i,y_i,z_i\in A$  such that
$$a = \alpha+\sum_i y_i[a,x_i]z_i,$$
and hence
$$a\otimes 1 =  \alpha\otimes 1+\sum (y_i\otimes 1)[a\otimes 1,x_i\otimes 1](z_i\otimes 1).$$
As $a\otimes 1 = t -s$ and  $s$ lies in $J$, the commutator of 
$a\otimes 1$ with any element in $T$ lies in $J$. Therefore, $a\otimes 1 -\alpha\otimes 1\in J$, and so
$$t = \alpha\otimes 1 + (a\otimes 1  -\alpha\otimes 1)+ s  \in  Z(A)\otimes 1 + J,$$
which proves that $t$ is centrally stable.
\end{proof}

\begin{remark}\label{rr}
Even if $A$ is not centrally stable, many (if not most) elements in $M_n(A)$ are still centrally stable. For example,
from the proof of Theorem \ref{t} it is evident that if $t=(a_{ij})\in M_n(A)$ is such that 
at least one of $a_{ii}$ is centrally  stable in $A$, then $t$ is centrally stable in $M_n(A)$ (the proof actually considers the case where $a_{nn}$ is centrally stable). On the other hand,
Lemma \ref{lik} shows that if
 $a=a_{11}=\dots=a_{nn}$ is not centrally stable in $A$ and $a_{ij}=0$ for $i\ne j$, then 
$t$ is not centrally stable in $M_n(A)$.
\end{remark}

From the discussion in  Remark \ref{rr} it   is evident that the sum of centrally stable elements is not always centrally stable.

\section{Centrally Stable Finite-Dimensional Algebras} \label{s4}

Let $A$ be a finite-dimensional algebra over a field $F$. By $\rad(A)$ we denote the radical of $A$; that is, $\rad(A)$ is a unique maximal nilpotent ideal of $A$. If $A$ is semisimple
 (i.e., $\rad(A)=0$), then Wedderburn's Theorem tells us that $A$ is isomorphic to a finite direct product of simple algebras of the form $M_{n_i}(D_i)$ where $n_i\ge 1$ and $D_i$ is a division algebra over $F$ (see, e.g., \cite[Theorem 2.64]{INCA}). Combining this with Example  \ref{et1} and Proposition \ref{dirsum}, we see that 
every finite-dimensional semisimple algebra is centrally stable. Under a mild assumption that $F$ is perfect, we will be able to determine  the structure  
of an arbitrary  centrally stable finite-dimensional unital algebra over $F$ (Theorem \ref{tmain}). Before proceeding to the proof of this result, we record a few  simple results
that concern algebras that are not necessarily unital.

\begin{proposition}
A nonzero finite-dimensional  algebra $A$ with zero center is not centrally stable.  
\end{proposition}

\begin{proof} As 
 $A^n=0$ obviously implies $A^{n-1}\subseteq Z(A)$, 
 $A$ is not a nilpotent algebra.
Thus, $\rad(A)$ is a proper ideal of $A$ and hence $A/\rad(A)$ is a nonzero semisimple algebra. In particular, $A/\rad(A)$ has a nonzero center, and so
  $A$ cannot be centrally stable. 
\end{proof}

\begin{proposition}\label{prad} Suppose an algebra $A$ contains a nilpotent ideal $N$ such that $A=Z(A)+N$. Then the only centrally stable elements in $A$ are the central elements.  
\end{proposition}

\begin{proof}
Suppose, on the contrary, that $A$ contains a non-central centrally stable element. Since $Z(A)+N=A$, then so does $N$. As $N$ is nilpotent, there exists a  positive integer $j$ such that
$N^j$ contains a non-central centrally stable element, but $N^{j+1}$ does not. Let $r\in N^j$ be such an element. By Proposition \ref{pe},
 $$r\in Z(A)+{\rm Id}([r,A]) = Z(A)+{\rm Id}([r,N]),$$
so  ${\rm Id}([r,N])$ also contains a  non-central centrally stable element. 
However,  ${\rm Id}([r,N])\subseteq N^{j+1}$, so this is a contradiction.
\end{proof}

For an illustration of Proposition \ref{prad}, see Example \ref{e33}.

\begin{corollary}\label{rad}
Let $A$ be a finite-dimensional algebra. If $A/\rad(A)$ is commutative and $A$ is not commutative, then $A$ is not centrally stable. 
\end{corollary}

\begin{proof}
By Proposition \ref{prad}, we may assume that $A\ne Z(A)+ \rad(A)$. If $a\in A$ is centrally stable, 
then, by Proposition \ref{pe}, 
$$a\in Z(A) +{\rm Id}([a,A]) \subseteq Z(A)+\rad(A).$$
Thus, only the elements in $Z(A)+\rad(A)$ may be centrally stable.
\end{proof}

Let us also record the following immediate corollary to Proposition \ref{prad}.
\begin{corollary}\label{nilp} A nilpotent algebra is centrally stable if and only if it is commutative.
\end{corollary}

\begin{remark} \label{r11}
If $C$ is a commutative finite-dimensional  unital algebra, then, by Corollary \ref{cc},  $A=M_n(C)$ is a centrally stable algebra. 
Note that 
$\rad(A)=M_n(\rad(C))$  is not commutative, and hence not centrally stable, as long as $n\ge 2$ and $\rad(C)^2 \ne 0$. This proves  the following two facts:
\begin{enumerate}
\item[(a)] An ideal of a centrally stable algebra may not be centrally stable.
\item[(b)] The algebra of $n\times n$ matrices over a commutative algebra without unity may not be centrally stable.
\end{enumerate}
 \end{remark}

We now turn towards our main goal, the proof of Theorem \ref{tmain}. Let us begin with  two examples  which helped us to conjecture the crucial lemma
following it.

\begin{example}\label{exg}
Let $D$ be a central division $F$-algebra, and let $A$ be the subalgebra of $M_2(D)$ consisting of matrices of the form  $\left[\begin{smallmatrix} a & b  \cr
0 & a \cr
 \end{smallmatrix} \right]$ with $a,b\in D$. Note that $V=\left[\begin{smallmatrix} 0 & D  \cr
0 & 0 \cr
 \end{smallmatrix} \right]$ is the only proper nonzero ideal of $A$ and  $A/V\cong D$. 
In particular, $A/V$ is a central algebra which implies that
 $A$ is centrally stable. Let us point out that $A$ itself is not central since 
$\left[\begin{smallmatrix} 0 & 1  \cr
0 & 0 \cr
 \end{smallmatrix} \right]\in Z(A)$.
\end{example}

\begin{example} \label{exh}
The real vector space $V=\mathbb C$ becomes a  $(\mathbb C,\mathbb C)$-bimodule by defining
$\lambda \cdot v = \lambda v$ and $v\cdot \mu = \overline{\mu} v$ for all $\lambda,\mu\in\mathbb C$ and $v\in V$. Let $A$ be the real algebra
of all  
 matrices  of the form  $\left[\begin{smallmatrix} \lambda & v  \cr
0 & \lambda \cr
 \end{smallmatrix} \right]$ with $\lambda\in \mathbb C$ and $v\in V$.
Observe that  $Z(A)$ consists of real scalar multiples of the identity (so $A$ is a central $\mathbb R$-algebra) and that $\rad(A)$ consists of matrices of the form $\left[\begin{smallmatrix} 0 & v  \cr
0 & 0 \cr
 \end{smallmatrix} \right]$ with $v \in V$. 
Note that $A/\rad(A)\cong \mathbb C$, so $A/\rad(A)$ is a proper field extension of the base field. 
 Therefore, $A$ is not centrally stable. We also remark that both $\rad(A)$ and $A/\rad(A)$ are commutative, and hence centrally stable algebras.
\end{example}

Example \ref{exg} may be viewed as an illustration of the next lemma, while Example \ref{exh} indicates the necessity of its assumptions.

\begin{lemma}\label{lll}Let $D$ be a finite-dimensional division $F$-algebra and let $V$ be a nonzero unital $(D,D)$-bimodule such that $cx=xc$ for all $c \in Z(D)$ and $x \in V$. If $V$ is finite-dimensional as a right vector space over $D$, then there exists a basis of  the right vector space $V$ over $D$ which is contained in the center of $V$.
In particular, $V$ has a nonzero center.
\end{lemma}

\begin{proof}
Let $\{v_1,\dots,v_n\}$ be a basis of the right vector space $V$ over $D$. Given any $x\in D$ and $1\le j\le n$, there exists a unique $\varphi_{ij}(x)\in D$ such that
\begin{equation} \label{e1}xv_j = \sum_{i=1}^n v_i \varphi_{ij}(x).\end{equation}
Hence, 
$$(xy)v_j = \sum_{i=1}^n v_i \varphi_{ij}(xy)$$
for all $x,y\in D$.
On the other hand,
\begin{align*}
(xy)v_j = x(yv_j) = \sum_{k=1}^n (xv_k) \varphi_{kj}(y)
= \sum_{i=1}^n \sum_{k=1}^n v_i \varphi_{ik}(x)\varphi_{kj}(y).
\end{align*}
Comparing both expressions, we get
$$\varphi_{ij}(xy)=\sum_{k=1}^n \varphi_{ik}(x)\varphi_{kj}(y).$$
Since, by assumption, elements of $V$ commute with elements of $Z(D)$, it follows from (\ref{e1}) that for all $x \in V$, $c \in Z(D)$ and $1\le j\le n$, 
$$\sum_{i=1}^nv_i\varphi_{ij}(cx)=(cx)v_j=c(xv_j)=c \left(\sum_{i=1}^n v_i \varphi_{ij}(x)\right)= \sum_{i=1}^nv_i(c\varphi_{ij}(x)).$$
This shows that for all  $1 \leq i,j\leq n$, the maps $x\mapsto \varphi_{ij}(x)$ are $Z(D)$-linear. Observe also that \eqref{e1} implies that $\varphi_{ij}(1)=\delta_{ij}$. Hence, if we consider $D$ and $M_n(D)$ as algebras over $Z(D)$, we conclude that 
the map $\varphi:D\to M_n(D)$ given by
$$\varphi(x)= \big(\varphi_{ij}(x)\big)$$
is a $Z(D)$-algebra homomorphism (which sends the unity $1$ of $D$ to the unity $I_n$ of $M_n(D)$). We may consider $D$ as a subalgebra of $M_n(D)$ (by identifying $x\in D$ with the scalar matrix $xI_n\in M_n(D)$). Since $D$ is finite-dimensional as an $F$-algebra, it  is also finite-dimensional as a $Z(D)$-algebra. The Skolem-Noether Theorem \cite[Theorem 4.46]{INCA} therefore tells us that there exists an invertible matrix $A\in M_n(D)$ such that $$\varphi(x)= A (xI_n) A^{-1}$$ for every $x\in D$. 
Using this along with \eqref{e1}, we see that $$v = [v_1\,v_2\,\dots\,v_n]\in V^n$$ satisfies
$$ \big((xI_n) v^t\big)^t = v \varphi(x) = vA (xI_n) A^{-1}$$
for all $x\in D$. Multiplying from the right by $A$, we obtain
$$ \big((xI_n) v^t\big)^tA= vA (xI_n).$$
Writing
$A=(a_{ij})$, we thus have 
$$[xv_1\,xv_2\,\dots\,xv_n](a_{ij}) = [v_1\,v_2\,\dots\,v_n](a_{ij}x),$$
that is,
$$x\Big(\sum_{i=1}^n v_ia_{ij} \Big) = \Big(\sum_{i=1}^n v_ia_{ij} \Big)x$$
for all $x\in D$ and all $1\le j\le n$. This means that each $u_j=\sum_{i=1}^n v_ia_{ij} $ lies in the center of $V$. Further, since $A$ is invertible, it is clear that $\{u_1, \dots, u_n\}$ forms a basis of the right vector space $V$ over $D$. 
\end{proof}



\begin{lemma}\label{mt2}
Let $A$ be a finite-dimensional unital algebra whose center is a field. Then $A$ is centrally stable if and only if $A$ is simple.
\end{lemma}

\begin{proof} We only have to prove the ``only if'' part. So, assume that 
 $A$  is centrally stable. Since $Z(A)$ is a field,  $A$ may be regarded as a central algebra over $Z(A)$.
The central stability assumption implies that 
 $A/\rad(A)$ is also central. Thus, $A/\rad(A)$ is a semisimple algebra having no nontrivial central idempotents. Therefore, by  
Wedderburn's Theorem, there exist a finite-dimensional  division  algebra $D$ over $F$ and 
a positive integer $n$ such that $A/\rad(A)=M_n(D)$.  

It remains to show that $\rad(A)=0$. Assume, on the contrary,  that $\rad(A)\neq 0$. Using \cite[Proposition 13.13]{Row}, it follows
 that there exist a unital algebra $B$ and an ideal $R$ of $B$ such that 
$$A= M_n(B),\,\,\,\,B/R=D,\,\,\,\,\mbox{and}\,\,\,\,\rad(A)=M_n(R).$$ 
Since $A$ is centrally stable,  Proposition \ref{l}\,(a) implies that so is $B$. Further, since  $\rad(A)$ is nilpotent, so is $R$.
 Let $m\ge 2$ be such that  $V= R^{m-1}\ne 0$ and $R^m=0$.
 For each $b + R\in D$ and $v \in V$, define
$$(b+ R) v = bv\,\,\,\,\mbox{and}\,\,\,\, v(b+R)=vb.$$
Since $RV=VR=R^m=0$, these are well-defined operations from $D\times V$ to $V$ and $V\times D$ to $V$, respectively, which make $V$  a unital $(D,D)$-bimodule. Of course,
as a finite-dimensional vector space over $F$, $V$ is also  a finite-dimensional right (as well as left) vector space over $D$. Since $B$ is centrally stable, any central element of $D$ can be written as $c+R$ for some $c \in Z(B)$.
This implies that $(c+R)v=v(c+R)$ for all $c+R \in Z(D)$ and $v \in V$. Therefore, by Lemma \ref{lll}, there exists a $v\ne 0$ in $V$ such that $(b+ R) v = v(b+R)$, and hence $bv = vb$, for all $b\in B$. That is, $v$ lies in the center of $B$. Consequently, $vI_n\in \rad(A)$ lies in the center of $A$, which contradicts the assumption that $Z(A)$ is a field.
\end{proof}

Incidentally, Proposition \ref{End} shows that Lemma \ref{mt2} does not hold for infinite-dimensional algebras.

\begin{lemma}\label{tt}
Let $A$ be a  finite-dimensional unital algebra 
whose center does not contain nonzero nilpotent elements. Then $A$ is centrally stable if and only if it is semisimple.
\end{lemma}

\begin{proof}
Our assumption on the center can be stated as that $Z(A)$ is a semisimple algebra, and hence isomorphic to $F_1\times\cdots\times F_r$ where each $F_i$ is a finite field extension of the base field $F$.
Therefore, there exist orthogonal central idempotents $e_1,\dots,e_r$ such that $e_1+\dots+e_r =1$ and hence $A=e_1A\oplus\dots\oplus e_rA$, and $Z(e_iA)=F_ie_i$.
 Assume now that $A$ is centrally stable. Then each $e_iA$ is centrally stable (by Proposition \ref{dirsum}) and so Lemma \ref{mt2} implies that it is simple. Consequently, $A$ is semisimple.
The converse statement is trivial.
\end{proof}

Another way of stating Lemma \ref{tt} is as follows: If $A$ is a finite-dimensional centrally stable unital algebra such that $\rad(A)\ne 0$, then $Z(A)\cap \rad(A)\ne 0$.
The next theorem gives a sharper conclusion.

\begin{theorem}\label{mm}
Let $A$ be a  finite-dimensional  unital algebra. If $A$ is centrally stable, then $\rad(A)={\rm Id}\big(Z(A)\cap \rad(A)\big)$. 
\end{theorem}

\begin{proof}
Suppose, on the contrary,  that $I={\rm Id}\big(Z(A)\cap \rad(A)\big)$ is a proper subset of $\rad(A)$. Then $A/I$ is a centrally stable algebra with nonzero radical. 
By Lemma \ref{tt}, there exists an $a\in A\setminus I$ such that  $a+I\in Z(A/I)\cap \rad(A/I)$. Note that this implies 
${\rm Id}([a,A]) \subseteq I$ and $a\in \rad(A)$. Since $a$ is centrally stable, Proposition \ref{pe} implies that there exist  a $c\in Z(A)$ and a $u\in I$  such that $a=c+u$. However, then $c=a-u\in 
\rad(A)$ and hence $c\in I$. Consequently, $a=c+u\in I$, which is a contradiction.
\end{proof}

We are now in a position to prove our main result.

\begin{theorem}\label{tmain}
Let $A$ be a  finite-dimensional  unital algebra over a perfect field $F$. The following conditions are equivalent:
\begin{enumerate}
\item[{\rm (i)}] 
 $A$ is centrally stable. 
\item[{\rm (ii)}] $\rad(A)={\rm Id}\big(Z(A)\cap \rad(A)\big)$.
\item[{\rm (iii)}]  There exist finite field extensions $F_1,\dots,F_r$ of $F$, commutative unital $F_i$-algebras $C_1,\dots, C_r$, and central simple $F_i$-algebras $A_1, \dots, A_r$
 such that $A\cong (C_1\otimes_{F_1} A_1)\times \cdots\times  (C_r\otimes_{F_r} A_r)$.
\end{enumerate}
\end{theorem}

\begin{proof} (i)$\implies$(ii). Theorem \ref{mm}.

(ii)$\implies$(iii). Since $F$ is perfect, we may use
  Wedderburn's Principal Theorem (see, e.g., \cite[p. 191]{Row}) which tells us that there exists a semisimple algebra $A_0$ contained in $A$ such that
$A$ is the vector space direct sum of $A_0$  and rad$(A)$. Let us write $C$ for $Z(A)\cap \rad(A)$. According to our assumption,  every 
 element  in $\rad(A)$ is  a sum of elements of the form $c(a+r)$ where $c\in C$, $a\in A_0$ and $r\in \rad(A)$. Now, $r$ can be further  written as a sum of elements of the form
 $c'(a'+r')$. Since  $$c(a+c'(a'+r'))=ca + cc'a' + cc'r' \in CA_0 + C^2 \rad(A)$$  and $C^n \subseteq \rad(A)^n =0$ for some $n$, continuing in this vein we arrive at
\begin{equation}\label{zz}\rad(A)= CA_0.\end{equation}

As $A_0$ is a  semisimple $F$-algebra, 
there exist idempotents $e_1,\dots,e_r\in Z(A_0)$ such that $e_1 + \dots + e_r =1$ and $A_i=e_iA_0$ is a simple $F$-algebra. Denote by $F_i$ the center of 
of $e_iA_0$. Of course, $F_i$ is a finite field extension of $F$ and $A_i$ is a central simple $F_i$-algebra. From \eqref{zz} it is evident that $Z(A_0)\subseteq Z(A)$, so $e_i$ are central idempotents in $A$ and $A$ is the direct sum of ideals $e_1A,\dots,e_rA$. 
Write $C_i$ for $F_i + e_iC$. Note that $F_i=F_ie_i\subseteq Z(A_0)\subseteq Z(A)$ and that $C_i$ is a commutative $F_i$-algebra.
Let us prove  that $$e_iA\cong C_i\otimes_{F_i} A_i.$$ It is clear that $e_iA$ is generated by $C_i\cup A_i$ and that the elements of $C_i$ commute with the elements of $A_i$. 
According to \cite[Lemma 4.24]{INCA}, it remains to show that
whenever
$a_1,\dots,a_m\in A_i$ are linearly independent over $F_i$ and $c_1,\dots,c_m\in C_i$ are such that 
$c_1a_1+\dots + c_ma_m=0$, then each $c_j=0$. To this end, we invoke the  Artin-Whaples Theorem which implies that 
for each $j$ there exists an $f_j\in M(A_i)$ such that $f_j(a_j)=e_j$ and $f_j(a_k)=0$ if $k\ne j$. Since 
$f_j$ obviously satisfies $f_j(ca)=cf_j(a)$ for all $c\in C_i$ and $a\in A_i$, it follows that
$$0= f_j(c_1a_1+\dots + c_ma_m)=c_je_j=c_j.$$ Thus, $e_iA\cong C_i\otimes_{F_i} A_i$ and 
$$A=e_1A\oplus\cdots \oplus e_rA\cong (C_1\otimes_{F_1} A_1)\times \cdots\times  (C_r\otimes_{F_r} A_r).$$

(iii)$\implies$(i). Propositions \ref{dirsum} and \ref{l}\,(b).
\end{proof}

Since, by Wedderburn's Theorem, each $A_i$ is isomorphic to $D_i\otimes_{F_i}M_{n_i}(F_i)$ for some $n_i\ge 1$ and a  central division $F_i$-algebra $D_i$,
we may state (iii) as 
$$A\cong M_{n_1}(C_1\otimes_{F_1} D_1)\times \cdots\times  M_{n_r}(C_r\otimes_{F_r} D_r).$$
If $F$ is algebraically closed, then $D_i=F_i=F$ and the theorem reduces to the following corollary.

\begin{corollary}
A  finite-dimensional  unital algebra $A$ over an algebraically closed field $F$ is centrally stable if and only if there exist 
positive integers $n_1,\dots,n_r$ and commutative unital $F$-algebras $C_1,\dots,C_r$ such that $$A\cong M_{n_1}(C_1)\times \cdots\times M_{n_r}(C_r).$$
\end{corollary}

\end{document}